\documentclass[11pt,a4paper]{amsart}
\usepackage[top=4cm,bottom=4cm,left=3.5cm,right=3.5cm]{geometry}
\usepackage{subcaption}
\usepackage{tabularx}
\usepackage{amsmath,amssymb,amsthm}
\usepackage[english]{babel}
\usepackage{mathtools}
\usepackage{enumerate}
\usepackage[numbers]{natbib}
\usepackage[hidelinks]{hyperref}
\usepackage{xcolor}

\newcommand{\fp}{\mathfrak{p}}

\newcommand{\ri}{\mathfrak{o}}

\newcommand{\Kp}{K_\mathfrak p}

\newcommand{\Zp}{\mathfrak{o}_\mathfrak{p}}
 \DeclarePairedDelimiter\card{\lvert}{\rvert}

\newtheorem{theorem}{Theorem}

\newtheorem{lemma}[theorem]{Lemma}
\newtheorem{proposition}[theorem]{Proposition}
\newtheorem{corollary}[theorem]{Corollary}
\theoremstyle{definition}

\newtheorem{definition}[theorem]{Definition}
\theoremstyle{remark}
\newtheorem{remark}[theorem]{Remark}

\numberwithin{equation}{section}

\begin{document}
	\title[Local to global principle over number fields for higher moments]{Local to global principle over number fields for higher moments}

	\author[G. Micheli]{Giacomo Micheli}
	\address{Department of Mathematics\\
		University of South Florida\\
		Tampa, FL 33620, United States of America
	}
	\email{gmicheli@usf.edu}
	
	\author[S. Schraven]{Severin Schraven}
	\address{Department of Mathematics, University of British Columbia, 1984 Mathematics Road, Vancouver, BC V6T 1Z2, Canada
	}
	\email{sschraven@math.ubc.ca}

	\author[S. Tinani]{Simran Tinani}
	\address{Institute of Mathematics\\
		University of Zurich\\
		Winterthurerstrasse 190\\
		8057 Zurich, Switzerland
	}
	\email{simran.tinani@math.uzh.ch}

	\author[V. Weger]{Violetta Weger}
	\address{Department of Electrical and Computer Engineering\\
		Technical University of Munich\\
		Theresienstrasse 90\\
		80333 Munich, Germany
	}
	\email{violetta.weger@tum.de}

	\subjclass[2010]{}
	
	
	\maketitle
	
	\begin{abstract}
	The local to global principle for densities is a very convenient tool proposed by Poonen and Stoll to compute the density of a given subset of the integers.
	In this paper we provide an effective criterion to find all higher moments of the density (e.g. the mean, the variance) of a subset of a finite dimensional free module over the ring of algebraic integers of a number field. More precisely, we provide a local to global principle that allows the computation of all higher moments corresponding to the density, over a general number field $K$. This work advances the understanding of local to global principles for density computations in two ways: on one hand, it extends a result of  Bright, Browning and Loughran, where they provide the local to global principle for densities over number fields; on the other hand, it extends the recent result on a  local to global principle for expected values over the integers  to both the ring of algebraic integers and to moments higher than the expected value.  
	To show how effective and applicable our method is, we compute the density, mean and variance of Eisenstein polynomials and shifted Eisenstein polynomials over number fields. This extends (and fully covers) results in the literature that were obtained with ad-hoc methods.
	\end{abstract}
	
	\section{Introduction}
%To measure the relative size of subsets is often of interest and mostly invokes the use of probabilities. Classically this is done using a uniform probability distribution. However,   to measure  subsets of the integers a different method is required, for example the density.  
 For a positive integer $d$, the density of a set $T \subseteq \mathbb{Z}^d$ is defined as 
	$$\rho(T) = \lim\limits_{H \to \infty} \frac{ \mid T \cap [-H,H[^d \mid}{ (2H)^d},$$ if it exists. Computing the density of a subset $T$ of $\mathbb Z^n$ is a classical problem in number theory, as it provides an estimate on the relative size of $T$ and $\mathbb Z^n$

	A well-known result in this area is the density of coprime pairs, which has been computed by Mertens \cite{mertens1874ueber} and C\'esaro \cite{ce1,ce3} independently. This result has been generalized to coprime $m$-tuples by Nymann \cite{bib:nymann1972probability} and further to rectangular unimodular matrices by Micheli and Weger \cite{densitiesunimodular}.  Another interesting target set is the set of Eisenstein polynomials, where the results are due to Dubickas \cite{bib:PolyDub} in the case of monic polynomials and due to Heymann and Shparlinski \cite{bib:shparlinskiEisen} in the non-monic case. 
	In addition, Micheli and Schnyder computed the density of shifted and affine Eisenstein polynomials over $\mathbb Z$ in \cite{micheli2016densityeis}.
	Since the density can also be defined over a general number field, further generalizations of these results appeared (see \cite{bib:MicheliCesaro, densitiesunimodular}), using a pull-back function from the set of interest, i.e., the ring of algebraic integers to the integers.
	
An elegant tool to compute certain densities over the integers is  called the local to global principle. This principle was introduced by Poonen and Stoll \cite{poonenAnn} and uses a local characterization of the target set, i.e., over the $p$-adic integers.
	In fact, all of the above mentioned examples have a density that can be computed through this tool. The local to global principle has been generalized to number fields by Bright \emph{et al.} in \cite{bright2016failures} and to function fields by Micheli in \cite{functionfields}. 
	
	Recently, a new question regarding densities has arisen in \cite{primenumbereisenstein}: to compute the expected value and the variance of Eisenstein polynomials over the integers. More precisely, the question can be formulated as \emph{``On an average, for a random Eisenstein polynomial how many primes $p$ are such that this Eisenstein polynomial satisfies the criterion of Eisenstein for $p$?"} 
	This led to \cite{ltgm}, where a general definition of the expected value corresponding to the density is given and an addendum to the local to global principle over the integers is provided, which allows to compute these expected values corresponding to the density directly using the local characterization of the target set. In addition, in the thesis \cite{weger} the variance of Eisenstein polynomials and other target sets over the integers were determined using an adaption of this new local to global principle.
	
	This paper finally closes all the missing gaps in this historical overview, as we provide a final extension of the local to global principle over number fields, which allows to compute all higher moments of a target set. As an application of this result, we compute the missing density of Eisenstein polynomials over number fields, in addition to their mean and variance. 
	
	The paper is organized as follows: in Section \ref{sec:preliminaries} we recall the definition of density over the integers as well as over number fields  and restate the original local to global principle by Poonen and Stoll and its generalization to number fields by Bright \emph{et al.}
	In Section \ref{sec:highermoments} we give the definition of the expected value and higher moments of a system  and  present the main result, Theorem \ref{Thm:highermoments}, an extension of the local to global principle, which works over a general number field and allows to compute any higher moment. Finally, we compute the density, the mean and the variance of Eisenstein polynomials over number fields in Section \ref{sec:applications}.

	\section{Preliminaries}\label{sec:preliminaries}
For a positive integer $d$, the density of a set $T \subseteq \mathbb{Z}^d$ is defined by restricting to a $d$-dimensional cube of height $H$ and letting $H$ go to infinity.
	
	\begin{definition}[Density]
	    Let $d \in \mathbb{N}.$ The density of $T \subseteq \mathbb{Z}^d$ is defined to be 
	   $$\rho(T) = \lim\limits_{H \to \infty} \frac{ \mid T \cap [-H,H[^d \mid}{ (2H)^d},$$ if the limit exists.
	\end{definition}
	Similarly, one can define the upper density $\overline{\rho}$  and the lower density $\underline{\rho}$ using the limit superior and the limit inferior, respectively.

There exist various tools to compute the density of a set, one of the main tools is the local to global principle \cite{poonenAnn, bib:loctoglob} by Poonen and Stoll, which allows to compute  the density of certain sets by characterizing these sets   over the $p$-adic integers. 

	To present the principle we need to introduce the following notation.
For a set $S$ we denote by $2^S$ its powerset and by $S^C$ its complement. In addition, if $S$ is a subset of a metric space we denote by $\partial(S)$ its boundary. We denote by $\mathcal{P}$ the natural primes and for $p \in \mathcal{P}$ we denote by $\mathbb{Z}_p$ the $p$-adic integers. Let $M_{\mathbb{Q}} = \{ \infty \} \cup \mathcal{P}$ be the set of all places of $\mathbb{Q},$ where we denote by $\infty$ the unique Archimedean place of $\mathbb{Q}.$  Finally, for $d \in \mathbb{N},$ we denote by $\mu_\infty$ the Lebesgue measure on $\mathbb{R}^d$ and by $\mu_p$ the normalized Haar measure on $\mathbb{Z}_p^d.$

	\begin{theorem}[\text{\cite[Lemma 1]{bib:loctoglob}}]\label{Thm:original_poonen} \label{PoonenStoll}
		Let $d$ be a positive integer. Let $U_{\infty} \subseteq \mathbb{R}^d$, such that $\mathbb{R}_{\geq0} \cdot U_{\infty} = U_{\infty}$ and $\mu_{\infty}(\partial(U_{\infty}))=0.$ Let $s_{\infty}= \frac{1}{2^d}\mu_{\infty}(U_{\infty} \cap [-1,1]^d)$. 
		For each prime $p$, let $U_p \subseteq \mathbb{Z}_p^d$ be such that $\mu_p(\partial(U_p)) =0$ and define $s_p = \mu_p(U_p)$. 
		Define the following map 
		\begin{eqnarray*}
			P: \mathbb{Z}^d   &\rightarrow &  2^{M_{\mathbb{Q}}}, \\
			a  &\mapsto & \left\{ \nu \in M_{\mathbb{Q}} \mid a \in U_{\nu} \right\}.
		\end{eqnarray*}
		If the following is satisfied:
		\begin{equation}\label{originalcond}
		\lim_{M \rightarrow \infty} \bar{\rho}\left( \left\{ a \in \mathbb{Z}^d \mid a \in U_p \  \text{for some prime} \ p > M \right\} \right)=0,
		\end{equation}
		then:
		\begin{itemize}
			\item[i)] $\sum\limits_{\nu \in M_{\mathbb{Q}}} s_{\nu}$ converges.
			\item[ii)] For $\mathcal{S} \subseteq 2^{M_{\mathbb{Q}}},$  $\rho(P^{-1}(\mathcal{S}))$ exists, and defines a measure on $2^{M_{\mathbb{Q}}}$.
			\item[iii)] For each finite set $S \in 2^{M_{\mathbb{Q}}}$, we have that
			\begin{equation*}
			\rho(P^{-1}(\{S\})) = \prod_{\nu \in S} s_{\nu} \prod_{\nu \not\in S} (1-s_{\nu}), 
			\end{equation*}
			and if $\mathcal{S}$ consists of infinite subsets of $2^{M_{\mathbb{Q}}}$, then $\rho(P^{-1}(\mathcal{S}))=0.$
		\end{itemize}
	\end{theorem}

	To show that Condition \eqref{originalcond} is satisfied, one can often apply the following lemma.
	
	\begin{lemma}[ \text{\cite[Lemma 2]{poonenAnn}}]\label{lem:original}
	  Let $d$ and $M$ be positive integers. Let $f,g \in \mathbb{Z}[x_1, \ldots, x_d]$ be relatively prime and define
	  $$S_M(f,g) = \{ a \in \mathbb{Z}^d \mid f(a) \equiv  g(a) \equiv 0 \mod p \ \text{for some prime} \ p >M\}.$$ Then
	  $$\lim\limits_{M \to \infty} \overline{\rho}(S_M(f,g)) =0.$$
	\end{lemma}

The notion of density, as well as the local to global principle can be generalized to number fields. For this we will first recall some basics. 

	Let $K$ be a number field and denote by $\ri$ its ring of integers. For any non-zero prime ideal $\mathfrak{p}$ of $\ri$ we denote the completion of $K$ with respect to the non-Archimedean place $\mathfrak{p}$ by $\Kp$. Furthermore, we denote by $\Zp$ the ring of integers of $\Kp$. 
	Note that $\Zp$ is compact in the subspace topology of $\Kp$. Hence, there exists a unique normalized Haar measure on $\Zp$, which we denote by $\nu_\mathfrak{p}$. We define the degree of the number field $K$ to be the index $k=[K: \mathbb{Q}]$, which is finite by assumption. Thus, one has that $\ri$ is isomorphic to $\mathbb{Z}^k$ as a $\mathbb{Z}$-module. Let $\mathbb{E}=\{ e_1, \dots, e_k \}$ be an integral basis for $\ri$, then we can define for $H\in \mathbb{N}$ the following set
	\begin{equation*}
	O(H,\mathbb{E}) = \left\{ \sum_{j=1}^k a_j e_j \in \ri \ : \ a_j \in [-H, H[ \cap \mathbb{Z} \right\}.
	\end{equation*}
	Let $d\in \mathbb{N}$ and $ T \subseteq \ri^d$, then we define the upper, respectively the lower density on the number field $K$ as
	\begin{align*}
	\overline{\rho}_\mathbb{E}(T) = \limsup_{H\rightarrow \infty} \frac{\vert T \cap O(H, \mathbb{E})^d \vert}{(2H)^{dk}}
	\end{align*}
	and
		\begin{align*}
	 \underline{\rho}_\mathbb{E}(T) = \liminf_{H\rightarrow \infty} \frac{\vert T \cap  O(H, \mathbb{E})^d  \vert}{(2H)^{dk}},
	\end{align*}
	respectively.
	If $\overline{\rho}_\mathbb{E}(T) = \underline{\rho}_\mathbb{E}(T)$,  we also define the density of $T$ with respect to the basis $\mathbb{E}$ as
	\begin{align*}
	\rho_\mathbb{E}(T) = \lim_{H\rightarrow \infty} \frac{\vert T \cap O(H, \mathbb{E})^d  \vert}{(2H)^{dk}}.
	\end{align*}
	Furthermore, in the case where $\rho_\mathbb{E}(T)$ is independent of the integral basis $\mathbb{E}$,  we simply write $\rho(T)$.
	For a non-zero prime ideal $\mathfrak{p} \subset \ri$, we say that $\mathfrak{p}$ is lying above $p$ for some $p \in \mathcal{P}$, if  $\mathfrak{p} \cap \mathbb{Z}= p \mathbb{Z}$ and we write $\mathfrak{p} \vert p$.
	Finally, we denote by $\mathcal{P}_K$ all non-zero prime ideals in $\ri$. 
	
	In \cite{bright2016failures}, Bright \emph{et al.} provide a generalization of the original local to global principle to number fields. We will state the result translated to our setting.  Let us define the $\mathbb{R}$-algebra $k_\infty =\ri \otimes_\mathbb{Z} \mathbb{R}$, which has covolume $\card{\Delta_K}^{1/2}$, where $\Delta_K$ denotes the discriminant of $K$ and $k_\infty$ is endowed with the Haar measure $\nu_{\infty}$, see \cite[Proposition I.5.2]{neu}.
	That is, for some set $X \subseteq k_\infty$ we have
	$$\nu_\infty(X)=  \frac{1}{\card{ \Delta_K}^{1/2}} \mu_\infty (f(X)),$$ where $\mu_\infty$ denotes the Lebesgue measure and  $f$ is the isomorphism $ K \otimes_\mathbb{Q} \mathbb{R} \to \mathbb{R}^{k}$ from \cite[Proposition I.5.1]{neu}.
	Let us denote $\mathcal{P}_K \cup \{ \infty \} $ by $N_K$. 
	
	\begin{theorem}[ \text{\cite[Proposition 3.2]{bright2016failures} }] \label{PoonenStollNumbfields}
	Let $K$ be a number field and $k = [K: \mathbb{Q}].$ Let $\mathbb{E}$ be an integral basis of $\ri$  and let $d$ be a positive integer.
		 
		For each $\mathfrak{p} \in \mathcal{P}_K$, let $U_\mathfrak{p} \subseteq \Zp^d$ be such that $\nu_\mathfrak{p}(\partial( U_\mathfrak{p})) = 0$ and define $s_\mathfrak{p}= \nu_\mathfrak{p}(U_\mathfrak{p}).$  Let $U_\infty \subseteq k_\infty^d$, such that $\mu_\infty(\partial(U_\infty))=0$ and $\mathbb{R}_{\geq 0} \cdot U_\infty  = U_\infty$. We define $s_\infty= \frac{1}{2^{dk}}\nu_\infty\left(U_\infty \cap [-1,1] \cdot O(1, \mathbb{E})^d \right)$ 
		Define the following map 
		\begin{align*}
		P: \ri^d & \rightarrow 2^{N_K}, \\ a & \mapsto \{ \eta \in N_K \ \mid \ a \in U_\eta\}.
		\end{align*}
		If the following is satisfied:
		\begin{equation} \label{tail}
		\lim_{M \rightarrow \infty} \overline{\rho}_\mathbb{E} \left( \ri^d \cap \bigcup_{p>M} \bigcup_{\mathfrak{p}\in \mathcal{P}_K, \  \mathfrak{p} \vert p} U_\mathfrak{p} \right) = 0,
		\end{equation}
		then:
		\begin{enumerate}
			\item $\sum\limits_{\eta \in N_K} s_\eta$ converges.\\
			\item For $\mathcal{S}\subseteq 2^{N_K}$, $m(\mathcal{S}) \coloneqq \rho_\mathbb{E}(P^{-1}(\mathcal{S}))$  exists and  defines a measure on $2^{N_K}$. \\
			\item The measure $m$ is concentrated at the finite subsets of $N_K$. For each  finite set $S \in 2^{N_K}$, we have that
			\begin{equation} \label{prod formula}
			m(\{ S\}) = \prod_{\eta \in S} s_\eta \prod_{\eta \notin S} (1-s_\eta)
			\end{equation}
			and if $\mathcal{S}$ consists of infinite subsets of $2^{N_K}$ then $m(\mathcal{S})=0$.
		\end{enumerate}
	\end{theorem}
	\begin{remark} \label{independent}
		Note that for any two integral bases $\mathbb{E}_1, \mathbb{E}_2$ there exists an integer $c$ such that for all $H\in \mathbb{N}$ holds $O(H,\mathbb{E}_1) \subseteq O(cH, \mathbb{E}_2)$, and thus if \eqref{tail} holds for one integral basis, then it holds for all integral bases. 
% 		Thus, a posteriori $m$ is independent of the integral basis $\mathbb{E}$ due to  \eqref{prod formula}.
	\end{remark}
	
	In addition, we have a similar result to Lemma \ref{lem:original} over number fields, in order to show that Condition \eqref{tail} is satisfied.
	For $M$ a positive integer and $\mathfrak{p} \in \mathcal{P}_K$, consider the unique prime integer $p \in \mathcal{P}$ with $\mathfrak{p} \mid p$. We write  $\mathfrak{p} \succ M$ if and only if   $p>M$. 
	Similarly one defines $\mathfrak{p} \preceq M$. In addition, we will assume that $\infty \preceq M$ for all $M \in \mathbb{N}$.

	\begin{lemma} \label{showdens}
			Let $d$ and $M$ be positive integers. Let $f,g \in \ri[x_1, \ldots, x_d]$ be relatively prime. Define
			\begin{equation*}
			S_M(f,g) = \left\{ a \in \ri^d \mid  f(a) \equiv g(a) \equiv 0 \mod \mathfrak{p} \ \text{for some prime ideal} \  \mathfrak{p} \succ  M \right\},
			\end{equation*}
			then we have for every integral basis $\mathbb{E}$ of $\ri$
			\begin{equation*}
			\lim_{M \rightarrow \infty} \bar{\rho}_\mathbb{E}(S_M(f,g)) = 0. 
			\end{equation*}
		\end{lemma}
	\begin{proof}
		This follows directly from \cite[Lemma 3.3]{bright2016failures} applied to the subscheme defined by $f=0=g$.
	\end{proof}

	\section{Higher moments}\label{sec:highermoments}

	In \cite{ltgm} the authors generalized Theorem \ref{PoonenStoll}, the local to global principle over the integers, to expected values. We will now generalize Theorem \ref{PoonenStollNumbfields}, the local to global principle over number fields, to higher moments.
	
	Before we introduce the definition of expected values of systems $(U_\eta)_{\eta \in N_K}$, let us notice that for a general number field $K$, $\eta=\infty$ does no longer correspond to an archimedean place. Still we find it useful to include this possiblity to modifity the box $O(H, \mathbb{E})^d$. Even though $N_K$ does no longer correspond to the set of all places, we will keep using the same notation in order to stay consistent with our previous paper \cite{ltgm}.
	Note that the set of elements living in infinitely many $U_\eta$, i.e.,
	$$I= \{ A \in \ri^d \mid A \in U_\eta \ \text{for infinitely many} \ \eta \in N_K\}$$
	has density zero; this follows directly from Condition \eqref{tail}.
	Let us denote by $O(H, \mathbb{E})_I = O(H, \mathbb{E}) \setminus I.$

	\begin{definition}\label{meandef}
		Let $d$  be a positive integer and assume that $U_\eta$ satisfies the assumptions of Theorem \ref{PoonenStollNumbfields}  for all $\eta \in N_K$, then we define \textit{the expected value
		of the system} $( U_\eta)_{\eta \in N_K}$ to be
		\begin{equation*}
		\mu_\mathbb{E}  = \lim\limits_{H \to \infty} \displaystyle{\frac{\sum\limits_{A \in O(H, \mathbb{E})_I^d  } \card{\{ \eta \in N_K \mid A \in U_{\eta} \} }}{  (2H)^{kd}} },
		\end{equation*}
		if it exists. More generally, for any non-negative integer $n$ we define \textit{the $n$-th moment of the system $(U_\eta)_{\eta\in N_K}$} to be
		\begin{equation*}
		\mu_{n,\mathbb{E}}  = \lim\limits_{H \to \infty} \displaystyle{\frac{\sum\limits_{A \in O(H,\mathbb{E})_I^d  } \card{ \{ \eta \in N_K \mid A \in U_{\eta} \}}^n }{  (2H)^{kd}} },
		\end{equation*}
		if it exists.
	\end{definition}
	This limit essentially gives the expected value of the number of   $\eta$, such that a random element in $\ri^d$ is in $U_\eta$.

	\begin{definition}\label{corr}
		For a set $T\subseteq \ri^d$, we say that a \emph{system $(U_\eta)_{\eta \in N_K}$ corresponds to $T$}, if Condition \eqref{tail} is satisfied and $T^C= P^{-1}(\{\emptyset\})$.
	\end{definition}
	
	As in \cite{ltgm} we can restrict Definition  \ref{meandef} to subsets of $ \ri^d$, i.e.,
	we define the $n$-th moment of the system $(U_\eta)_{\eta \in N_K}$ restricted to $ T \subseteq \ri^d$ to be 
	\begin{equation*}
	\mu_{n,T, \mathbb{E}}  = \lim\limits_{H \to \infty} \displaystyle{\frac{\sum\limits_{A \in O(H, \mathbb{E})_I^d\cap T  } \card{ \{ \eta \in N_K \mid A \in U_{\eta} \} }^n }{ \card{  O(H, \mathbb{E})_I^d \cap T }} },
	\end{equation*}
	if it exists. We will write $\mu_{T,\mathbb{E}}$ for $\mu_{1,T, \mathbb{E}}$ and if it does not depend on the integral basis $\mathbb{E}$, we will just write $\mu_{n,T}$, respectively $\mu_T$ for $\mu_{1,T}$.
	Note, that this is similar to the conditional expected value.

	For any non-negative integer $n$, one can easily pass from $\mu_{n,\mathbb{E}}$ to $\mu_{n,T, \mathbb{E}}$ and vice versa. The proof is the same as in \cite[Lemma 11]{ltgm}.
	
	\begin{lemma}\label{passtoT}
		If the density of $T$ with respect to $\mathbb{E}$ exists and is non-zero and $T$ is such that $T^C \subseteq P^{-1}(\{ \emptyset\})$, then $\mu_{n,T, \mathbb{E}}$ exists if and only if $\mu_{n, \mathbb{E}}$ exist. In addition,   we have that $\mu_{n, \mathbb{E}}=\mu_{n,T, \mathbb{E}} \cdot \rho_\mathbb{E}(T)$.
	\end{lemma}
	
	We have the following straightforward corollary from Theorem \ref{PoonenStollNumbfields}. 
	
	\begin{corollary}\label{denofUW}
		Let $\eta_1, \dots, \eta_n \in N_K$ with $ \eta_i \neq \eta_j$ for $i\neq j$ and let $U_{\eta_j}$ be chosen as in Theorem \ref{PoonenStollNumbfields}, then 
		$$\rho_\mathbb{E}\left( \bigcap_{j=1}^n\left(U_{\eta_j} \cap \ri^d  \right) \right) 
		%= \prod_{j=1}^n \nu_{\eta_j}(U_{\eta_j} \cap \ri^d) 
		= \prod_{j=1}^n s_{\eta_j},$$
		where again  $s_\eta= \nu_\eta(U_\eta).$
	\end{corollary}

	With the above corollary we are able to prove the following generalized version of Theorem \ref{PoonenStollNumbfields}, the main result of this paper.

	\begin{theorem}\label{Thm:highermoments}
		Let $d$ and $n$ be positive integers. Let $K$ be a number field with $k = [K: \mathbb{Q}]$, $\ri$ its ring of integers and $\mathbb{E}$ an integral basis of $\ri$. For each  $\mathfrak{p} \in \mathcal{P}_K$, let $U_\mathfrak{p} \subseteq \Zp^d$ be such that $\nu_\mathfrak{p}(\partial(U_\mathfrak{p})) =0$ and define $s_\mathfrak{p} = \nu_\mathfrak{p}(U_\mathfrak{p})$.  Let $U_\infty \subseteq k_\infty^d$, such that $\mu_\infty(\partial(U_\infty))=0$ and $\mathbb{R}_{\geq 0} \cdot U_\infty  = U_\infty$. We define $s_\infty= \frac{1}{2^{dk}}\nu_\infty\left(U_\infty \cap [-1,1] \cdot O(1, \mathbb{E})^d \right)$.
% 		Define the following map 
% 		\begin{eqnarray*}
% 			P: \ri^d   &\rightarrow &  2^{N_K}, \\
% 			a  &\mapsto & \left\{ \eta \in N_K \mid a \in U_{\eta} \right\}.
% 		\end{eqnarray*}
		If \begin{equation} \label{densitycond}
		\lim_{M \rightarrow \infty} \overline{\rho}_\mathbb{E} \left( \ri^d \cap \bigcup_{p>M} \bigcup_{\mathfrak{p}\in \mathcal{P}_K, \ \mathfrak{p} \vert p} U_\mathfrak{p} \right) = 0,
		\end{equation}
		is satisfied and for some $\alpha \in [0, \infty]$ there exist absolute constants $c', c \in \mathbb{Z}$, such that for all $H\geq 1$ and for all $A \in O(H,\mathbb{E})_I^d$ one has that
		\begin{equation}\label{newcond}
		\left\vert \left\{ \mathfrak{p}\in \mathcal{P}_K  \mid \mathfrak{p} \succ c'H^\alpha,  A \in U_\mathfrak{p} \cap O(H, \mathbb{E})_I^d  \right\} \right\vert <c
		\end{equation} 
		and  that there exists a sequence  $(v_\mathfrak{p})_{\mathfrak{p}\in \mathcal{P}_K}$, such that for all $\mathfrak{p}_1, \dots, \mathfrak{p}_n \preceq c'H^\alpha$ 
		one has that
		\begin{eqnarray}
		\left\vert \bigcap_{j=1}^n U_{\mathfrak{p}_j} \cap O(H, \mathbb{E})_I^d \right\vert & \leq (2H)^{kd} \prod\limits_{j=1}^{n} v_{\mathfrak{p}_j}, 
		\label{newcond2} \\
		\sum_{\mathfrak{p} \in \mathcal{P}_K} v_\mathfrak{p} & \text{converges}, \label{newcond3}
		\end{eqnarray}
		then it follows that
		\begin{align*}
		\mu_{n, \mathbb{E}}  &=  \lim\limits_{H \to \infty} \displaystyle{\frac{\sum\limits_{A \in O(H, \mathbb{E})_I^d
				  } \card{ \{ \eta \in N_K \mid A \in U_{\eta} \}}^n }{  (2H)^{kd}} }
		\end{align*}
		exists and $\mu_{n, \mathbb{E}}<\infty$. For $\tau\in \mathbb{N}^n$ with $\sum\limits_{j=1}^n j \tau_j=n$, we define $\ell(\tau) = \sum\limits_{j=1}^n \tau_j$ and denote by $c(\tau)$ the number of partitions of $\{1, \dots, n\}$ which contain exactly $\tau_j$ sets of cardinality $j$. Then we have the formula
		\begin{equation}\label{mean}
		\mu_{n, \mathbb{E}} = \sum_{\tau \in \mathbb{N}^n, \ \sum_{j=1}^n j \tau_j =n} c(\tau) \sum_{\substack{\eta_1, \dots, \eta_{\ell(\tau)}\in N_K  \\ \forall i<j \in \{1, \dots, \ell(\tau)\},  \ \eta_i \neq \eta_j}} \prod\limits_{m=1}^{\ell(\tau)} s_{\eta_m},
		\end{equation}
	\end{theorem}
	\begin{remark}
	    Note that similarly to Remark \ref{independent} we get that the conditions are satisfied for all integral bases (with possibly different constants) as soon as they are satisfied for one particular integral basis.  Also, the expression in \eqref{mean} is independent of the integral basis for $U_\infty = \emptyset$ (as only $s_\infty$ depends on $\mathbb{E}$). Note that the dependence on $\mathbb{E}$ is not just a technicality, but reflects the fact that the basis alters the way the density of $U_\infty$ is measured (as it is a cone and not a lattice, its density is not invariant under $\mathbb{Z}$-module isomorphisms). We could easily develop our theorem in the same setting as in \cite{bright2016failures}, where the box $O(H,\mathbb{E})^d$ is replaced by $H \Omega_\infty \cap \ri^d$ for some bounded set $\Omega_\infty \subset k_\infty^d$ with $\mu_\infty(\partial \Omega_\infty)=0$ and $\mu_\infty(\Omega_\infty)>0$. The proof would be the same and the results would only differ by normalizing constants. 
	    
	    Finally, if the conditions of Theorem \ref{Thm:highermoments} are satisfied for some integer $n \geq 1$, then the same holds true for any positive integer $m\leq n$.

	\end{remark}
	\begin{proof}
		For $A\in \ri^d$ and $\eta \in N_K$, we define
		\begin{align*}
		\tau(A,\eta) = \begin{cases}
		1,& A\in U_\eta,\\
		0,& \text{else}.
		\end{cases}
		\end{align*}
		For $M\in \mathbb{N},$ we have that
		$$\sum\limits_{A\in O(H, \mathbb{E})_I^d} \frac{\left( \sum\limits_{\eta \in N_K} \tau(A,\eta) \right)^n}{(2H)^{kd}}
		=\sum\limits_{j=0}^n \binom{n}{j}  R_j(M,H),$$
		where for all $j \in \{0, \ldots, n\}$, we define
		\begin{align*}
		R_j(M,H) \coloneqq \sum_{A\in O(H, \mathbb{E})_I^d} \frac{\left( \sum\limits_{\mathfrak{p}\in \mathcal{P}_K, \ \mathfrak{p} \succ M} \tau(A,\mathfrak{p}) \right)^{n-j} \left( \sum\limits_{\eta \in N_K, \ \eta \preceq M} \tau(A,\eta)\right)^j}{(2H)^{kd}}.
		\end{align*}
		First we show that for all $j\in \{0, \dots, n-1\}$ the terms $R_j(M,H)$ are negligible for $M$ going to infinity. We define
		\begin{align*}
		\ell_{A,H} \coloneqq \left\vert \left\{ \mathfrak{p}\in \mathcal{P}_K  \mid \mathfrak{p} \succ c'H^\alpha, A \in U_\mathfrak{p} \cap O(H, \mathbb{E})_I^d  \right\} \right\vert.
		\end{align*}
		Then by \eqref{newcond} there exists $c>0$ such that for all $A\in \ri^d$ and all $H \geq 1$ holds $\ell_{A,H} \leq c$. Thus, we get that
		\begin{align*}
	S_n(M & ,H):=  \sum\limits_{A\in O(H, \mathbb{E})_I^d} \frac{\left( \sum\limits_{\mathfrak{p}\in \mathcal{P}_K, \ \mathfrak{p} \succ M} \tau(A,\mathfrak{p}) \right)^n}{(2H)^{kd}} \\
		= & \sum\limits_{i=0}^n \binom{n}{i} \sum\limits_{\substack{A\in O(H, \mathbb{E})_I^d \\ A \in \bigcup_{\mathfrak{p}_1, \dots, \mathfrak{p}_n \succ M} \bigcap_{j=1}^n U_{\mathfrak{p}_j}}} \hspace{-1cm}  \frac{ \vert \{ (\mathfrak{p}_j)_{j=1}^n \in \mathcal{P}_K^n \ \vert \  M \prec \mathfrak{p}_1, \dots, \mathfrak{p}_i \prec c'H^\alpha \prec \mathfrak{p}_{i+1}, \dots, \mathfrak{p}_n  \} \vert }{(2H)^{kd}} \\
		\leq & \sum_{i=0}^n \binom{n}{i} \sum\limits_{\substack{A\in O(H, \mathbb{E})_I^d \\ A \in \bigcup_{\mathfrak{p}_1, \dots, \mathfrak{p}_n \succ M} \bigcap_{j=1}^n U_{\mathfrak{p}_j}}} \hspace{-0.8 cm} \frac{\ell_{A,H}^{n-i} \vert \{ (\mathfrak{p}_j)_{j=1}^i \in \mathcal{P}_K^i \ \vert \   M \prec \mathfrak{p}_1, \dots, \mathfrak{p}_i \prec c'H^\alpha  \} \vert }{(2H)^{kd}} \\
		\leq  & c^{n} \frac{\vert O(H, \mathbb{E})_I^d \cap \bigcup\limits_{\mathfrak{p}\in \mathcal{P}_K, \ \mathfrak{p} \succ M} U_\mathfrak{p} \vert}{(2H)^{kd}} 
		+\sum\limits_{i=1}^n \binom{n}{i} c^{n-i} \sum\limits_{\substack{(\mathfrak{p}_1, \dots, \mathfrak{p}_i)\in \mathcal{P}_K^i \\ M \prec \mathfrak{p}_1, \dots, \mathfrak{p}_i \prec c'H^\alpha}} \hspace{-0.4cm} \frac{\vert O(H, \mathbb{E})_I^d \cap \bigcap\limits_{j=1}^i U_{\mathfrak{p}_j} \vert}{(2H)^{kd}},
		\end{align*}
% 		where we have omitted to write $ (\mathfrak{p}_1, \dots, \mathfrak{p}_n) \in \mathcal{P}_K^n$ for space reasons.
		Using  \eqref{newcond2}, we further have that
		\begin{align*}
		S_n(M,H) \leq & c^{n} \frac{\vert O(H, \mathbb{E})_I^d \cap \bigcup\limits_{\mathfrak{p}\in \mathcal{P}_K, \ \mathfrak{p} \succ M} U_\mathfrak{p} \vert}{(2H)^{kd}} \\ &
		+ \sum\limits_{i=1}^n \binom{n}{i} c^{n-i} \sum\limits_{\substack{(\mathfrak{p}_1, \dots, \mathfrak{p}_i)\in \mathcal{P}_K^i \\ M \prec \mathfrak{p}_1, \dots, \mathfrak{p}_i \prec c'H^\alpha}} \left( \prod_{j=2}^{i} v_{\mathfrak{p}_j} \right) v_{\mathfrak{p}_1}^{n-i+1} \\
		\leq  & c^{n} \frac{\vert O(H, \mathbb{E})_I^d \cap \bigcup\limits_{\mathfrak{p}\in \mathcal{P}_K, \ \mathfrak{p} \succ M} U_\mathfrak{p} \vert}{(2H)^{kd}} \\ &
		+ \sum\limits_{i=1}^n \binom{n}{i} c^{n-i} \left( \sum_{\mathfrak{p} \in \mathcal{P}_K, \ \mathfrak{p} \succ M} v_\mathfrak{p} \right)^{i-1} \left( \sum_{\mathfrak{p} \in \mathcal{P}_K, \ \mathfrak{p} \succ M} v_\mathfrak{p}^{n-i+1} \right).
		\end{align*}
		
		This implies that
		\begin{align*}
		\limsup_{H\rightarrow \infty} S_n(M,H)
		\leq  & c^{n} \overline{\rho}_\mathbb{E}\left(  \ri^d \cap \bigcup\limits_{\mathfrak{p}\in \mathcal{P}_K, \ \mathfrak{p} \succ M} U_\mathfrak{p} \right)  \\ &
		+ \sum\limits_{i=1}^n \binom{n}{i} c^{n-i} \left( \sum\limits_{\mathfrak{p} \in \mathcal{P}_K, \ \mathfrak{p} \succ M} v_\mathfrak{p} \right)^{i-1} \left( \sum\limits_{\mathfrak{p} \in \mathcal{P}_K, \ \mathfrak{p} \succ M} v_\mathfrak{p}^{n-i+1} \right).
		\end{align*}
		Thus, we get from \eqref{densitycond} and \eqref{newcond3}
		\begin{equation} \label{Sn}
		\lim_{M \rightarrow \infty} \limsup_{H\rightarrow \infty} S_n(M,H) =0.
		\end{equation}
		On the other hand, \eqref{newcond2} and \eqref{newcond3} imply that
		\begin{equation} \label{Rn}
		\begin{split}
		R_n(M,H)&=\sum_{A\in O(H, \mathbb{E})_I^d} \frac{\left( \sum\limits_{\eta \preceq M } \tau(A,\eta) \right)^n}{(2H)^{kd}} \\
		&= \sum\limits_{\eta_1, \dots, \eta_n \preceq M } \sum\limits_{A\in O(H, \mathbb{E})_I^d} \frac{\prod\limits_{j=1}^n \tau(A,\eta_j)}{(2H)^{kd}} \\
		&= \sum\limits_{\eta_1, \dots, \eta_n \preceq M }  \frac{\vert O(H, \mathbb{E})_I^d \cap \bigcap\limits_{j=1}^n U_{\eta_j} \vert}{(2H)^{kd}} \\
		&\leq\sum\limits_{\eta_1, \dots, \eta_n \preceq M } \prod\limits_{j=1}^n v_{\eta_j} \\
		&\leq \left( \sum\limits_{\eta\in N_K} v_\eta \right)^n <\infty.
		\end{split}
		\end{equation}
		For $j\in \{1, \dots, n-1\}$ we have by Hölder's inequality with $p=\frac{n}{n-j}$ and $q=\frac{n}{j}$, that
		\begin{align*}
		R_j(M,H) \leq S_n(M,H)^{(n-j)/n}  R_n(M,H)^{j/n}.
		\end{align*}
		Thus, by \eqref{Sn} and \eqref{Rn} we get for all $j\in \{0, \dots, n-1\}$ 
		\begin{equation*}
		\lim_{M \rightarrow \infty} \limsup_{H\rightarrow \infty} R_j(M,H) = 0.
		\end{equation*}
		Hence, if we can show that $\lim\limits_{M \rightarrow \infty} \lim\limits_{H \to \infty} R_n(M,H)$ exists, then $\mu_{n, \mathbb{E}}$ exists as well and we have $$\mu_{n, \mathbb{E}} = \lim\limits_{M \rightarrow \infty} \lim\limits_{H\rightarrow \infty} R_n(M,H).$$ 
		
		For $\tau\in \mathbb{N}^n$ with $\sum\limits_{j=1}^n j \tau_j=n$, we define $\ell(\tau) = \sum_{j=1}^n \tau_j$ and denote by $c(\tau)$ the number of partitions of $\{1, \dots, n\}$ which contain exactly $\tau_j$ sets of cardinality $j$.
		Using Corollary \ref{denofUW}, we obtain

		\begin{align*}
		\mu_{n, \mathbb{E}} = & \lim\limits_{M \rightarrow \infty} \lim\limits_{H \to \infty} R_n(M,H)
		\\ = & \lim\limits_{M \rightarrow \infty} \lim\limits_{H \to \infty} \sum\limits_{\eta_1, \dots, \eta_n \preceq M } \frac{\vert O(H, \mathbb{E})_I^d \cap \bigcap\limits_{j=1}^n U_{\eta_j} \vert}{(2H)^{kd}} \\
		&= \lim\limits_{M \rightarrow \infty} \lim\limits_{H \to \infty} \sum\limits_{\tau \in \mathbb{N}^n, \ \sum_{j=1}^n j \tau_j =n} c(\tau) \sum\limits_{\substack{\eta_1, \dots, \eta_{\ell(\tau)} \preceq M  \\ \forall i<j\in \{1, \dots, \ell(\tau)\}, \ \eta_i \neq \eta_j}} \frac{\vert O(H, \mathbb{E})_I^d \cap \bigcap\limits_{j=1}^{\ell(\tau)} U_{\eta_j} \vert}{(2H)^{kd}} \\
		&= \lim\limits_{M \rightarrow \infty} \sum\limits_{\tau \in \mathbb{N}^n, \ \sum\limits_{j=1}^n j \tau_j =n} c(\tau) \sum\limits_{\substack{\eta_1, \dots, \eta_{\ell(\tau)}\preceq M  \\ \forall i<j\in \{1, \dots, \ell(\tau)\}, \ \eta_i \neq \eta_j}} \rho_\mathbb{E} \left( \bigcap\limits_{j=1}^{\ell(\tau)} U_{\eta_j} \cap \ri^d \right)  \\
		&= \lim\limits_{M \rightarrow \infty} \sum\limits_{\tau \in \mathbb{N}^n, \ \sum\limits_{j=1}^n j \tau_j =n} c(\tau) \sum\limits_{\substack{\eta_1, \dots, \eta_{\ell(\tau)} \preceq M  \\ \forall i<j\in \{1, \dots, \ell(\tau)\}, \ \eta_i \neq \eta_j}} \prod\limits_{j=1}^{\ell(\tau)} s_{\eta_j}  \\
		&=\sum\limits_{\tau \in \mathbb{N}^n,\ \sum_{j=1}^n j \tau_j =n} c(\tau) \sum\limits_{\substack{\eta_1, \dots, \eta_{\ell(\tau)} \in N_K \\ \forall i<j\in \{1, \dots, \ell(\tau)\} , \ \eta_i \neq \eta_j}} \prod\limits_{j=1}^{\ell(\tau)} s_{\eta_j}.
		\end{align*}
		Note that this last expression is finite. We have the crude estimate
		\begin{align*}
		    \mu_{n, \mathbb{E}} 
		    \leq 2^n n^n \left(1+ \sum\limits_{\eta\in N_K} s_\eta \right)^n < \infty
		\end{align*}
		by \eqref{newcond3}.
 
	\end{proof}
	
	\begin{remark}
		The conditions of Theorem \ref{PoonenStollNumbfields} also allow to conclude the existence of  all central moments up to order $n$. For $j\in \{1, \dots, n\}$ we have
		\begin{equation*}
		\lim\limits_{H \to \infty} \displaystyle{\frac{\sum\limits_{A \in O(H, \mathbb{E})_I^d  } \left(\card{\{ \eta \in N_K \mid A \in U_{\eta} \}  } - \mu_\mathbb{E}\right)^n }{  (2H)^{kd}} }
		= \sum_{j=0}^n \binom{n}{j} (-\mu_{\mathbb{E}})^{n-j} \mu_{j, \mathbb{E}}.
		\end{equation*}
		Now we briefly compute the variance
		\begin{align*}
		\sigma^2_\mathbb{E} = \lim\limits_{H \to \infty} \displaystyle{\frac{\sum\limits_{A \in O(H, \mathbb{E})_I^d  } (\card{ \{ \eta \in N_K \mid A \in U_{\eta} \} } - \mu_\mathbb{E})^2 }{  (2H)^{kd}} }.
		\end{align*}
		For $n=2$ we have two $\tau \in \mathbb{N}^2$ such that $\tau_1+\tau_2 =2$, namely $\tau^{(1)}= (2,0)$ and $\tau^{(2)} = (0,1)$. One readily computes $c\left(\tau^{(1)}\right)=1, \ell\left(\tau^{(1)}\right)=2$ and $c\left(\tau^{(2)}\right)=1, \ell\left(\tau^{(2)}\right)=1$. Thus, we get
		\begin{align*}
		\mu_{2, \mathbb{E}} 
		= \sum_{\eta_1, \eta_2 \in N_K,\ \eta_1 \neq \eta_2 } s_{\eta_1} s_{\eta_2} + \sum_{\eta \in N_K} s_\eta
		= \mu_\mathbb{E}^2 - \sum_{\eta \in N_K} s_\eta^2 + \mu_\mathbb{E}.
		\end{align*}
		Hence, we obtain
		\begin{equation*}
		\begin{split}
		\sigma^2_\mathbb{E} 
		= \mu_{2, \mathbb{E}} - 2 \mu_\mathbb{E}^2 + \mu_\mathbb{E}^2 
		= \mu_\mathbb{E} - \sum_{\eta \in N_K} s_\eta^2.
		\end{split}
		\end{equation*}
		Also for the variance we can restrict to subsets  $T \subseteq \ri^d, $ as
		
		\begin{align*}
		\sigma^2_{T, \mathbb{E}} = \lim\limits_{H \to \infty} \displaystyle{\frac{\sum\limits_{A \in O(H, \mathbb{E})_I^d \cap T } (\mid \{ \eta \in N_K \mid A \in U_{\eta} \}  \mid - \mu_\mathbb{E})^2 }{  \mid O(H, \mathbb{E})^d \cap T \mid} }.
		\end{align*}
		Due to Lemma \ref{passtoT}, if $T^C \subseteq P^{-1}(\{ \emptyset\})$ and the density of $T$ is non-zero, we get
		\begin{align*}
		    \sigma^2_T = \frac{1}{\rho(T)}\left( \mu^2 - \sum\limits_{\eta \in N_K} s_\eta^2 + \mu\right) - \frac{2\mu_T}{\rho(T)} \mu + \mu_T^2.
		\end{align*}
	\end{remark}

\section{Applications}\label{sec:applications}
	\subsection{Density computations}

In this section, we compute the densities of Eisenstein and shifted Eisenstein polynomials over number fields, using Theorem \ref{PoonenStollNumbfields}.

\begin{definition}\label{eisdefn_fp}
Let $\fp$ be a non-zero prime ideal of $\ri$. A polynomial $f(x) \in \ri[x]$ of degree $d$ represented by the tuple $(a_0, \ldots, a_{d-1}, a_d) \in \ri^{d+1}$ is said to be $\fp$-Eisenstein if
\begin{align*}
a_d \not\in \fp,  \ a_0 \not\in \fp^2  \text{ and} \ a_i \in \fp \; \forall i \in \{0, \ldots, d-1\}  .
\end{align*}
In addition, $f$ is said to be Eisenstein if there exists a prime ideal $\fp$ of $\ri$ such that $f(x)$ is $\fp$-Eisenstein.

\end{definition}

 Denote by ${\Gamma_\fp}^d$ the set of all $\fp$-Eisenstein polynomials of degree $d$ and by $\Gamma^d$ the set of all Eisenstein polynomials of degree $d$. We choose \begin{equation}\label{U_p}
     U_\fp = (\fp \Zp \setminus \fp^2 \Zp) \times (\fp \Zp)^{d-1} \times (\Zp \setminus \fp \Zp) \subseteq \Zp^{d+1}
 \end{equation} and  $U_\infty=\emptyset$.  Note that we have 
 \begin{equation}\label{P-eis}
     {\Gamma_\fp}^d = (\fp  \setminus \fp^2) \times \fp^{d-1} \times (\ri \setminus \fp) = \ri^{d+1} \cap U_\fp. 
 \end{equation}

By abuse of notation we will use the same symbol for an element of $\ri$ and its image in the ring of integers $\Zp$ in the completion $\Kp$. 

\begin{corollary}\label{reg_eis} Let $K$ be a number field,  $\ri$ its ring of integers and let $d\geq 2$ be an integer. The density of the set $\Gamma^d$ of Eisenstein polynomials of degree $d$ over $\ri$ is given by
        \begin{equation} \label{densityEisenstein}
            \rho\left(\Gamma^d\right)=  1-\prod\limits_{\mathfrak{p}\in \mathcal{P}_K}\left(1-\frac{(N(\mathfrak{p})-1)^2}{N(\mathfrak{p})^{d+2}}\right),
        \end{equation}
	where $N(\mathfrak{p})= \vert \ri/\mathfrak{p} \vert = p^{\deg(\mathfrak{p})}$, and $\deg(\mathfrak{p}) = [ \ri / \mathfrak{p} : \mathbb{F}_p]$.
	\end{corollary}
\begin{proof}
Recall the expressions for the sets ${\Gamma_\fp}^d$ and $U_\fp$ from equations \eqref{P-eis} and \eqref{U_p}.
	With the system of sets $(U_\eta)_{\eta \in N_K}$, consider the map $P$ defined as in Theorem \ref{PoonenStollNumbfields}. 
	Note that we have \[P^{-1}(\{\emptyset\}) = \{a \in \ri^{d+1} \mid a\not\in U_\fp \ \forall \fp \in \mathcal{P}_K\},\] thus 

	\[P^{-1}(\{\emptyset\})^C = \{a \in \ri^{d+1} \mid  \ a  \in \Gamma_{\fp}^{d} \ \text{ for some } \fp \in \mathcal{P}_K\} = \Gamma^d.\]
 
In other words, the system $(U_\eta)_{\eta \in N_K}$ corresponds (as per Definition \ref{corr}) to the set $\Gamma^d$ of Eisenstein polynomials of degree $d$.
	We clearly have that $\partial(U_\mathfrak{p}) =\emptyset$. Hence, in order to apply Theorem \ref{PoonenStollNumbfields} to this system, we only have to check \eqref{tail}. 
	
	Letting $f(x_1, \dots, x_{d+1})=x_1$ and $g(x_1, \dots, x_{d+1})=x_2$, we see that 
	$$ S_M(f,g) = \bigcup\limits_{\fp \succ M}  \left( \fp \times \fp \times \ri^{d-1} \right),$$
	and $$\ri^{d+1} \bigcap \left(\bigcup\limits_{\fp  \succ M}  U_\fp\right) = \bigcup\limits_{\fp  \succ M}  \Gamma_\fp^d \subseteq S_M(f,g).$$ Thus, the application of Lemma \ref{showdens} gives Condition \eqref{tail} in this case.
	One easily computes that $$\nu_\fp(U_\fp)= \dfrac{(p^{\deg(\fp)}-1)^2}{p^{\deg(\fp)(d+2)}}.$$ 
	
	Hence, applying Theorem \ref{PoonenStollNumbfields} yields 
	\begin{align*}
	    \rho(\Gamma^d) &= 1-m(\{\emptyset\}) = 1-\prod\limits_{\eta \in \emptyset} s_\eta \prod\limits_{\eta \not\in \emptyset}(1-s_\eta) \\
	    & =1- \prod\limits_{\fp \in \mathcal{P}_K}(1-\nu_{\fp}(U_\fp)) \\
	    &= 1-\prod\limits_{\fp\in \mathcal{P}_K}\left(1-\dfrac{(p^{\deg(\fp)}-1)^2}{p^{\deg(\fp)(d+2)}}\right) \\
 &=1-\prod\limits_{\fp\in \mathcal{P}_K}\left(1-\dfrac{(N(\fp)-1)^2}{N(\fp)^{d+2}}\right). 
	\end{align*}
\end{proof}

\begin{definition}
Let $\fp$ be a non-zero prime ideal in $\ri$ and $f \in \ri[x]$ be a monic polynomial of degree $d$. We call $f$ a shifted $\fp$-Eisenstein polynomial if there exists a $b \in \ri$, such that $f(x+b)$ is $\fp$-Eisenstein. In addition, $f$ is said to be a shifted Eisenstein polynomial if there exists a prime ideal $\fp$ such that $f$ is a shifted $\fp$-Eisenstein polynomial.
\end{definition}

We will denote by $\widetilde{\Gamma}^d$ and $\widetilde{\Gamma}_\fp^d$  the set of all shifted Eisenstein polynomials, respectively shifted $\fp$-Eisenstein polynomials, of degree $d$.
We will identify elements of $\Zp^{d+1}$ with monic polynomials of degree $d$ over $\Zp$. 

For $b \in \Zp^{d+1}$, we will denote by $\sigma_b$ the following map:
\begin{align*}\sigma_b : \Zp^{d+1} & \rightarrow \Zp^{d+1}, \\ f(x) &\mapsto f(x+b). 
\end{align*}

Note that $\sigma_b$ is an automorphism with inverse $\sigma_{-b}$. It is also clearly continuous and linear. Moreover, for any $b\in \ri$ we have that $\ri^{d+1}$ is invariant under $\sigma_b$, in fact $\sigma_b(\ri^{d+1}) = \ri^{d+1}$. By abuse of notation, for $b \in \ri$ we also denote the restricted maps $ \ri^{d+1} \rightarrow \ri^{d+1}$  by $\sigma_b$.

The following two results were developed following the methods in \cite{micheli2016densityeis}.

\begin{proposition}\label{binp}
Let $f(x) \in \ri[x]$ be a polynomial of degree $d$ and $b$ denote any element of $\ri$. Suppose that $f(x)$ is $\fp$-Eisenstein. Then $f(x+b)$ is $\fp$-Eisenstein if and only if $b \in \fp$. 
\end{proposition}

\begin{proof}
We first write 
$$f(x) = a_0 +a_1x + \cdots + a_{d-1}x^{d-1} +a_d x^d. $$
For $b \in \fp$, we have that 
\begin{align*}
   f (x+b) &=a_0+ a_1(x+b) + \cdots + a_{d-1}(x+b)^{d-1} + a_d (x+b)^d = \sum_{i=0}^d a_i' x^i,
\end{align*}
where $a_i' = \sum_{j=i}^d \binom{j}{i} a_j b^{j-i}.$

Since $f$ is $\fp$-Eisenstein, we must have that $a_j \equiv 0 \mod \fp$ for all $j \in \{0, \ldots, d-1\}$, $a_d \not\equiv 0 \mod \fp$, and $a_0 \not\equiv 0 \mod \fp^2$.

It follows directly that $a_i^\prime \equiv a_d b^{d-i} \mod \fp$ for all $i \in \{0, \ldots, d-1\}$.

Thus, $f(x+b)$ is $\fp$-Eisenstein if and only if  
\begin{align*}  b^{d-i}\equiv 0 \mod \fp \; \forall \; 0\leq i < d, \; \text{and} \; a_0' = \sum_{j=0}^d a_j b^j \not\equiv 0 \mod \fp^2, 
\end{align*} which is equivalent to 
\begin{align*}  b \equiv 0 \mod \fp, \; \text{and} \; {a_0}^\prime = \sum_{j=0}^d a_j b^j  \not\equiv 0 \mod \fp^2.
\end{align*}
Note, that when $b \equiv 0 \mod \fp $, ${a_0}^\prime \equiv a_0 \not\equiv 0 \mod \fp^2$ holds by default, since $f$ is $\fp$-Eisenstein. We have thus shown that $f(x+b)$ is $\fp$-Eisenstein if and only if $b \equiv 0 \mod \fp$, as required.
\end{proof}

\begin{corollary}\label{cordisjshifted}
Let $b_1$ and $b_2$ be elements of $\ri$ such that $b_1 \not\equiv b_2 \mod \fp$. Then \begin{equation*}
    \sigma_{b_1}^{-1}(\Gamma_\fp^d) \cap \sigma_{b_2}^{-1}(\Gamma_\fp^d) = \emptyset.
\end{equation*}
\end{corollary}
\begin{proof}
Let $b=b_2-b_1$, then it is easy to see that the statement of the corollary is equivalent to \[ \ \Gamma_\fp^d \cap \sigma_b^{-1}(\Gamma_\fp^d) =\emptyset. \]
The claim now follows immediately from Proposition \ref{binp}.
\end{proof}

\begin{theorem}\label{shiftedchar} Let $\fp$ be a non-zero prime ideal of $\ri$. We have the following decomposition for the set $\widetilde{\Gamma}_\fp^d$ of all shifted $\fp$-Eisenstein polynomials of degree $d$:
\[\widetilde{\Gamma}_\fp^d = \bigsqcup_{b+\fp \in  \ri / \fp}  \sigma_b^{-1}\left(\Gamma_\fp^d\right), \] where $b+\fp$ is an element in $\ri/\fp$.

\end{theorem}
\begin{proof}
The set of all shifted $\fp$-Eisenstein polynomials of degree $d$ can be written as the union \begin{equation}\label{union}
    \bigcup_{b\in\ri} \sigma_b^{-1}(\Gamma_\fp^d) = \bigsqcup_{b+\fp \in \ri /\fp} \bigcup_{c \in b+\fp } \sigma_c^{-1}(\Gamma_\fp^d), 
\end{equation}
where the fact that the first union is disjoint follows from Corollary \ref{cordisjshifted}. We also have, from Proposition \ref{binp}, that for all $c_1$ and $c_2$ satisfying $c_1 \equiv c_2 \mod \fp$, \[\sigma_{c_1}^{-1}(\Gamma_\fp^d) \cap \sigma_{c_2}^{-1}(\Gamma_\fp^d) = \sigma_{c_1}^{-1}(\Gamma_\fp^d) = \sigma_{c_2}^{-1}(\Gamma_\fp^d). \]
Thus, the latter union in  \eqref{union} is equal to any one of the sets, and we may write \[\widetilde{\Gamma}_\fp^d = \bigsqcup_{b+\fp \in  \ri / \fp}  \sigma_b^{-1}(\Gamma_\fp^d), \] where $b+\fp$ is an element in $\ri/\fp$.

\end{proof}

\begin{corollary} \label{Cor:densShifted} Let $K$ be a number field,  $\ri$ its ring of integers and let $d\geq 2$ be an integer. The density of the shifted Eisenstein polynomials of degree $d$ over $\ri$ is given by
		\[ \rho\left(\widetilde{\Gamma}^d\right)= 1-\prod\limits_{\mathfrak{p}\in \mathcal{P}_K}\left(1- \frac{(N(\mathfrak{p})-1)^2}{N(\mathfrak{p})^{d+1}}\right),\]
	where $N(\mathfrak{p})= \vert \ri/\mathfrak{p} \vert = p^{\deg(\mathfrak{p})}$, where $\deg(\mathfrak{p}) = [ \ri / \mathfrak{p} : \mathbb{F}_p]$.
	\end{corollary}
\begin{proof} We have, 
\[\widetilde{\Gamma}_{\fp}^{d} = \bigsqcup_{b+\fp \in  \ri / \fp}  \sigma_b^{-1}\left((\mathfrak{p} \setminus \mathfrak{p}^2 )\times \mathfrak{p}^{d-1} \times \ri \setminus \mathfrak{p}\right) = \bigsqcup_{b+\fp \in  \ri / \fp}  \sigma_b^{-1}\left({\Gamma_\fp}^d\right) , \] where $b+\fp$ is an element in $\ri/\fp$. 
Let $U_\fp$ be defined as before.

We set 
\begin{equation*}
     \widetilde{U}_\fp =  \bigsqcup\limits_{b+\fp \in  \ri / \fp} \sigma_b^{-1} \left(U_\fp\right) \subset {\Zp}^{d+1},
\end{equation*}
 and $\widetilde{U}_\infty = \emptyset$.

As noted before, $\sigma_b(\ri^{d+1})= \ri^{d+1}$ for any $b\in \ri$ and thus, we have $\widetilde{\Gamma}_\fp^d = \ri^{d+1} \cap \widetilde{U}_\fp$.
Now, with the system of sets $\widetilde{U}_\mathfrak{p}$ and the map $P$ defined as in Theorem \ref{PoonenStollNumbfields}, we have, as before, that $P^{-1}(\{\emptyset\})^C$ is the set $\widetilde{\Gamma}^d$ of shifted Eisenstein polynomials of degree $d$. Since $\Zp$ is compact, and ${\sigma_b}$ is a surjective and continuous endomorphism of $\Zp^{d+1}$, we have that for any $b \in \ri$, $\sigma_b$ preserves the Haar measure. Thus, \[\nu_{\fp}(\widetilde{U}_\fp) = \nu_\fp\left(\bigsqcup_{b+\fp \in  \ri / \fp}  \sigma_b^{-1}\left(U_\fp\right)\right) = \left \lvert {\ri}/{\fp}\right \rvert \nu_{\fp}({U}_\fp)= p^{\deg(\fp)}\cdot \nu_{\fp}({U}_\fp).\]

We first deal with the case $d\geq 3$ and verify condition \eqref{tail} for the system $(\widetilde{U}_\eta)_{\eta \in N_K}$. For this, we define polynomials $F, G \in \ri[x_0, \ldots, x_{d}]$ as \begin{align*}
    F(x_0, \ldots, x_{d})& = d^2 x_0 x_{d} -  x_1 x_{d-1},\\
    G(x_0, \ldots, x_{d})& = x_2 x_{d-2} - \binom{d}{2}^2 x_0 x_{d}.
\end{align*} Clearly, these are coprime for any value of $d\geq 3$. We claim that 
\begin{equation} \label{inclusion}
    \ri^{d+1} \bigcap \left(\bigcup\limits_{\fp  \succ M}  \widetilde{U}_\fp\right) = \bigcup\limits_{\fp  \succ M}  \widetilde{\Gamma}_\fp^d \subseteq S_M(F,G).
\end{equation}

Let $ a^\prime=(a_0^\prime, a_1^\prime, \ldots, a_{d}^\prime) \in \widetilde{\Gamma}_\fp^d$ for some $\fp \succ M$. Then, there exists $b \in \ri$ such that $\sigma_{-b}(a^\prime)=a=(a_0, \dots, a_d)$ is $\fp$-Eisenstein. 
This is equivalent to saying $a'=\sigma_b(a)$ and hence we have the relations $a_j^\prime = \sum_{i=j}^d \binom{i}{j}a_i b^{i-j}$, $a_i \equiv 0 \mod \fp$ for $i\in \{0, \dots, d-1\}$ and $a_d \not \in \fp$. Thus, $a_j^\prime =  \binom{d}{j}a_d b^{d-j} \mod \fp$.  Clearly,  $a_j^\prime a_{d-j}^\prime \equiv \binom{d}{j}^2 a_d^2 b^d \mod \fp$, and in particular \begin{align*} 
a_0^\prime a_{d}^\prime  &\equiv  a_d^2 b^d  \mod \fp, \\ 
a_1^\prime a_{d-1}^\prime & \equiv  a_d^2 b^d d^2 \mod \fp, \\ a_2^\prime a_{d-2}^\prime & \equiv  a_d^2 b^d \binom{d}{2}^2 \mod \fp.
\end{align*} Therefore, we get
$$F(a^\prime) \equiv G(a^\prime) \equiv 0 \mod \fp$$
or in other words $a^\prime \in S_M(F,G)$. This proves \eqref{inclusion}.

Finally, from Lemma \ref{showdens} we have 		\begin{equation*}
			\lim_{M \rightarrow \infty} \bar{\rho}_\mathbb{E}(S_M(F,G)) = 0 
			\end{equation*}
and thus
\begin{equation*}
			\lim_{M \rightarrow \infty} \bar{\rho}_\mathbb{E}\left(\ri^{d+1} \bigcap \left(\bigcup\limits_{\fp  \succ M}  \widetilde{U}_\fp\right)\right) = 0. 
\end{equation*}

 Thus, applying Theorem \ref{PoonenStollNumbfields} and using \eqref{U_p}, we get
 \begin{align*}
	    \rho\left(\widetilde{\Gamma}^d\right) = 1-m(\{\emptyset\}) 
	   &= 1-\prod\limits_{\fp \in \mathcal{P}_K}(1-\nu_{\fp}(\widetilde{U}_\fp)) \\
	    &= 1-\prod\limits_{\fp\in \mathcal{P}_K}\left(1-p^{\deg(\fp)}\cdot\dfrac{(p^{\deg(\fp)}-1)^2}{p^{\deg(\fp)(d+2)}}\right) \\
 &=1-\prod\limits_{\fp\in \mathcal{P}_K}\left(1- \dfrac{(N(\fp)-1)^2}{N(\fp)^{d+1}}\right).
 \end{align*}

Now we turn to the case $d=2$. We use the same strategy as in \cite[Prop. 10]{micheli2016densityeis} and fix some positive integer $M$ and some integral basis $\mathbb{E}$ of $\ri$. Then we consider the system
\begin{align*}
    V_\mathfrak{p} = \begin{cases} \widetilde{U}_\mathfrak{p},& \mathfrak{p}\preceq M,\\ \emptyset,& \text{else}, \end{cases}
\end{align*}
and $V_\infty = \emptyset$. The system $(V_\eta)_{\eta\in N_K}$ clearly satisfies the conditions of Theorem \ref{PoonenStollNumbfields} and thus
\begin{align} 
    1\geq \underline{\rho}_\mathbb{E}(\widetilde{\Gamma}^2) 
    \geq \underline{\rho}_\mathbb{E}\left( \bigcup_{\eta \in N_K} V_\eta \cap \ri^{3} \right) 
    &= 1 - \prod_{\mathfrak{p}\in \mathcal{P}_K, \ \mathfrak{p}\preceq M} \left( 1 - \frac{(N(\mathfrak{p})-1)^2}{N(\mathfrak{p})^3} \right) \nonumber \\
    &\geq 1 - \prod_{\mathfrak{p}\in \mathcal{P}_K, \ \mathfrak{p}\preceq M} \left( 1 - \frac{1}{2N(\mathfrak{p})} \right).   \label{densM}
\end{align}
By \cite[Ch. VIII, Theorem 6]{lang2013algebraic} the series $\sum_{\mathfrak{p}\in \mathcal{P}_K} N(\mathfrak{p})^{-1}$ diverges to infinity. Hence, the product  of \eqref{densM} goes to zero for $M \rightarrow \infty$. Thus, we get 
$$\rho(\widetilde{\Gamma}^2)=1 = 1- \prod_{\mathfrak{p}\in \mathcal{P}_K} \left(1- \frac{(N(\mathfrak{p})-1)^2}{N(\mathfrak{p})^{3}} \right).$$

\end{proof}

 \subsection{Computations of Higher Moments}
	In this section, we will apply Theorem \ref{Thm:highermoments} to compute the expected value and the variance of non-zero prime ideals $\mathfrak{p} \subset \ri$ over a general number field $K$ such that a polynomial of degree $d$ is $\mathfrak{p}$-Eisenstein. A minor modification yields the same claim for shifted Eisenstein polynomials.
 
	\begin{corollary}\label{cor:eisenstein} Let $d\geq 2$ be an integer and let $K$ be number field and $\Gamma^d$ be the Eisenstein polynomials of degree $d$. We associate the system $U_\mathfrak{p}= (\mathfrak{p}\Zp \setminus \mathfrak{p}^2 \Zp) \times (\mathfrak{p} \Zp)^{d-1} \times (\Zp \setminus \mathfrak{p} \Zp)$ and $U_\infty= \emptyset$. This system satifies the conditions of Theorem~\ref{Thm:highermoments} for any $n\in \mathbb{N}$.
	
	In particular, we have 
	\begin{align*}
	    \rho\left(\Gamma^d\right) &= \left(1-\prod\limits_{\mathfrak{p}\in \mathcal{P}_K}\left(1-\frac{(N(\mathfrak{p})-1)^2}{N(\mathfrak{p})^{d+2}}\right)\right), \\
	    \mu &= \sum\limits_{\mathfrak{p}\in \mathcal{P}_K}\frac{(N(\mathfrak{p})-1)^2}{N(\mathfrak{p})^{d+2}}, 
	    \quad \mu_{\Gamma^d} = \rho (\Gamma^d)^{-1} \mu ,
	\end{align*}
	where $N(\mathfrak{p})= \vert \ri/\mathfrak{p} \vert = p^{\deg(\mathfrak{p})}$ and $\deg(\mathfrak{p}) = [ \ri / \mathfrak{p} : \mathbb{F}_p]$. 
	 
	Furthermore, the restricted variance is given by
	\begin{align*}
	   \sigma^2_{\Gamma^d} = \frac{1}{\rho(\Gamma^d)}\left( \mu^2 - \sum\limits_{\eta \in \mathcal{P}_K} \left(1-\frac{(N(\mathfrak{p})-1)^2}{N(\mathfrak{p})^{d+2}}\right)^2 + \mu\right) - \frac{2\mu_{\Gamma^d}}{\rho(\Gamma^d)} \mu + \mu_{\Gamma^d}^2.
	\end{align*}
	\end{corollary}
	\begin{proof}

We start by noting that $\rho(\Gamma^d)$ was already computed in \eqref{densityEisenstein}.

As $\ri$ is a Dedekind domain, we get that the intersection of infinitely many prime ideals in $\ri$ is the zero ideal. In particular, we have that the intersection of infinitely many ideals of the form $\mathfrak{p} \setminus \mathfrak{p}^2$ must be the empty set and hence, $I=\emptyset$ and $O(H,\mathbb{E})_I = O(H, \mathbb{E})$.

We need to check the assumptions of Theorem \ref{Thm:highermoments}. Condition \eqref{densitycond}   follows directly from Lemma \ref{showdens} applied to the polynomials $f(x_0, \dots, x_d)=x_0$ and $g(x_0, \dots, x_d)=x_1$.

Next we show that \eqref{newcond} is satisfied for any $\alpha>0$. Let $A=(A_0, \dots, A_d)\in \ri^{d+1}$, then we define $f_A(x)= \sum_{j=0}^d A_{d-j} x^j$ and we denote by $\operatorname{disc}(f_A)$ the discriminant of $f_A$. If $\{ \mathfrak{p} \in \mathcal{P}_K \ \mid \ A \in U_\mathfrak{p} \} = \emptyset$, then \eqref{newcond} is trivially satisfied. Thus, we will now assume that $\{ \mathfrak{p} \in \mathcal{P}_K \ \mid \ A \in U_\mathfrak{p} \} \neq \emptyset$.

As the discriminant is the Sylvester matrix of the resultant of $f_A$ and $f_A'$, we get (as $A_d$ will always be multiplied by some $A_j$ with $j\in \{0, \dots, d-1\}$)
\begin{align*}
    \operatorname{disc}(f_A) \in \bigcap_{\mathfrak{p}\in \mathcal{P}_K, A \in U_\mathfrak{p}} \mathfrak{p}.
\end{align*}
Combining this with the observation that $\operatorname{disc}(f_A)\neq 0$ as $f_A$ is irreducible (as $A$ is contained in some $U_\mathfrak{p}$, $f_A$ satisfies the criterion of Eisenstein), we get
\begin{equation} \label{low}
\begin{split}
    \vert N_{K/\mathbb{Q}}(\operatorname{disc}(f_A)) \vert 
    &=  N(\langle\operatorname{disc}(f_A)\rangle_{\ri})  \\
    &\geq \prod_{\mathfrak{p}\in \mathcal{P}_K, \ A \in U_\mathfrak{p}} \vert N(\mathfrak{p}) \vert 
    \geq H^{\alpha \vert \{ \mathfrak{p}\in \mathcal{P}_K \ \mid \ \mathfrak{p}\succeq H^\alpha, A \in U_\mathfrak{p} \} \vert },
\end{split}
\end{equation}
where we denote by $N$ the absolute norm and $N_{K/\mathbb{Q}}$ the ideal norm on $\ri$. 

On the other hand, using that the resultant is a homogeneous polynomial of degree $2d-2$ and $N_{K/\mathbb{Q}}(\operatorname{disc}(f_A)) = \operatorname{det}_\mathbb{Z}(\ri \rightarrow \ri, x \mapsto \operatorname{disc}(f_A)x)$, there exists a constant $c>0$, depending only on $\mathbb{E}$, such that for all $H\geq 1$ and all $A\in \ri^{d+1}$ holds 
\begin{equation} \label{high}
    \vert N_{K/\mathbb{Q}}(\operatorname{disc}(f_A)) \vert \leq c H^{(2d-2)k},
\end{equation}
where $k=[K:\mathbb{Q}]$.
Thus, combining \eqref{low} and \eqref{high} and taking the logarithm, we obtain for $H\geq 1$
\begin{align*}
    \alpha \vert \{ \mathfrak{p}\in \mathcal{P}_K \ \mid \ \mathfrak{p}\succeq H^\alpha, A \in U_\mathfrak{p} \} \vert
    \leq \ln(c) + (2d-2) k.
\end{align*}
Hence, Condition \eqref{newcond} holds for any choice of $\alpha>0$.

Now we check Conditions \eqref{newcond2} and \eqref{newcond3}. Let $\mathbb{E}$ be an integral basis of $\ri$, $k=[K:\mathbb{Q}]$ and $\varphi: \ri \rightarrow \mathbb{Z}^k$ be the isomorphism of $\mathbb{Z}$-modules induced by the basis $\mathbb{E}$. We get
\begin{align*}
    \frac{ \left\vert \bigcap_{j=1}^n U_{\mathfrak{p}_j} \cap O(H, \mathbb{E})^{d+1} \right\vert }{ (2H)^{(d+1)k}}
\leq \left( \frac{ \left\vert \bigcap_{j=1}^n \mathfrak{p}_j \cap O(H, \mathbb{E}) \right\vert}{(2H)^k} \right)^d.
\end{align*}
Hence, we are only interested in ideals of the form $ \bigcap_{j=1}^n \mathfrak{p}_j$. In fact, as $\varphi$ preserves densities of lattices, it is enough to consider its image in $\mathbb{Z}^k$ under $\varphi$. By a similar argument as for \cite[Proposition 1]{baake2000diffraction} one can show that for any $n\in \mathbb{N}$ there exists a universal constant $c>0$ such that for any lattice $\Gamma \subseteq \mathbb{R}^n$ of full rank and all $H\in \mathbb{N}_{\geq 1}$ holds
\begin{equation*} 
  	        \left\vert \frac{\vert \Gamma \cap [-H, H)^n\vert}{(2H)^n} - \rho(\Gamma) \right\vert
 	        \leq c \rho(\Gamma) \left( \frac{\operatorname{diam}(\Gamma)}{H} + \left( \frac{\operatorname{diam}(\Gamma)}{H} \right)^2 \right),
  	    \end{equation*}
  	   where $\operatorname{diam}(\Gamma)$ denotes the diameter of the fundamental domain of $\Gamma$.

Hence, if $\operatorname{diam}( \varphi( \bigcap_{j=1}^n \mathfrak{p}_j)) \leq cH$, then we can estimate 

\begin{equation} \label{auxEisenstein}
\begin{split}
      &\frac{\left\vert \bigcap_{j=1}^n \mathfrak{p}_j \cap O(H, \mathbb{E}) \right\vert}{(2H)^k} 
=  \frac{\left\vert \varphi \left(\bigcap_{j=1}^n \mathfrak{p}_j\right) \cap [-H,H)^k \right\vert}{(2H)^k} 
\leq  c' \rho\left( \varphi \left(\bigcap_{j=1}^n \mathfrak{p}_j \right)\right) \\
&= c' \rho \left( \bigcap_{j=1}^n \mathfrak{p}_j \right) 
= c' \prod_{j=1}^n [\ri : \mathfrak{p}_j]^{-1}
= c' \prod_{j=1}^n p_j^{-\operatorname{deg}(\mathfrak{p}_j)}
\leq c'  \prod_{j=1}^n \frac{1}{p_j}.
\end{split}
\end{equation}
Thus, for $d\geq 2$ we can pick $v_\mathfrak{p} = (1+c')/p^d$ and the series converges. All we need to show is that there exists some $\alpha>0$ and some universal constant $C>0$ such that for $\mathfrak{p}_1, \dots, \mathfrak{p}_n \preceq H^\alpha$ holds
\begin{align*}
    \operatorname{diam}\left( \varphi \left(\bigcap_{j=1}^n \mathfrak{p}_j \right)\right) \leq C H.
\end{align*}
However, by Minkowski's Second Theorem \cite[Chapter VIII.2, Theorem 1]{cassels2012introduction} we have for any ideal $\mathfrak{a} \in \ri$  
\begin{align*}
   \operatorname{diam}(\varphi(\mathfrak{a}))\leq  \rho(\varphi(\mathfrak{a}))^{-1}
   = \rho(\mathfrak{a})^{-1}
   = [\ri : \mathfrak{a}].
\end{align*}
By the Chinese Remainder Theorem we get
\begin{align*}
    \operatorname{diam}\left(\varphi\left( \bigcap_{j=1}^n \mathfrak{p}_j\right) \right)
\leq [\ri :   \bigcap_{j=1}^n \mathfrak{p}_j]
=  \prod_{j=1}^n [\ri : \mathfrak{p}_j]
\leq  \prod_{j=1}^n p_j^k,
\end{align*}
Hence, if  $\mathfrak{p}_1, \dots, \mathfrak{p}_n \preceq H^\alpha$, then we get
\begin{align*}
    \operatorname{diam}\left( \bigcap_{j=1}^n \mathfrak{p}_j \right)
\leq  H^{\alpha n k}.
\end{align*}

Choosing $\alpha = 1/(nk)$, we obtain
\begin{align*}
    \operatorname{diam}\left( \bigcap_{j=1}^n \mathfrak{p}_j \right)
\leq  H.
\end{align*}

\end{proof}

The argument of Corollary \ref{cor:eisenstein} can be generalized to shifted Eisenstein polynomials as well.

\begin{corollary} \label{cor:shiftEisenstein}
Let $d\geq 3$ be an integer and let $K$ be number field and $\widetilde{\Gamma}^d$ be the shifted Eisenstein polynomials of degree $d$. We associate the system 
$$\widetilde{U}_\mathfrak{p}= \bigcup_{b+\mathfrak{p}\in \ri /\mathfrak{p}}\sigma_b^{-1}\left( (\mathfrak{p}\Zp \setminus \mathfrak{p}^2 \Zp) \times (\mathfrak{p} \Zp)^{d-1} \times (\Zp \setminus \mathfrak{p} \Zp)\right)$$
and $\widetilde{U}_\infty= \emptyset$. 
Then the system $(\widetilde{U}_\mathfrak{p})_{\mathfrak{p}\in N_K}$ satisfies the conditions of Theorem~\ref{Thm:highermoments} for any $n\in \mathbb{N}$.
	
	In particular, we have 
	\begin{align*}
	    \rho(\widetilde{\Gamma}^d) &= \left(1-\prod\limits_{\mathfrak{p}\in \mathcal{P}_K}\left(1-\frac{(N(\mathfrak{p})-1)^2}{N(\mathfrak{p})^{d+1}}\right)\right), \\
	    \mu &= \sum\limits_{\mathfrak{p}\in \mathcal{P}_K}\frac{(N(\mathfrak{p})-1)^2}{N(\mathfrak{p})^{d+1}}, 
	    \quad \mu_{\widetilde{\Gamma}^d} = \rho (\widetilde{\Gamma}^d)^{-1} \mu ,
	\end{align*}
	where $N(\mathfrak{p})= \vert \ri/\mathfrak{p} \vert = p^{\deg(\mathfrak{p})}$ and $\deg(\mathfrak{p}) = [ \ri / \mathfrak{p} : \mathbb{F}_p]$. 
	 
	Furthermore, the restricted variance is given by
	\begin{align*}
	   \sigma^2_{\widetilde{\Gamma}^d} = \frac{1}{\rho(\widetilde{\Gamma}^d)}\left( \mu^2 - \sum\limits_{\eta \in \mathcal{P}_K} \left(1-\frac{(N(\mathfrak{p})-1)^2}{N(\mathfrak{p})^{d+1}}\right)^2 + \mu\right) - \frac{2\mu_{\widetilde{\Gamma}^d}}{\rho(\widetilde{\Gamma}^d)} \mu + \mu_{\widetilde{\Gamma}^d}^2.
	\end{align*} 
\end{corollary}
\begin{proof}
     Recall that we already have verified \eqref{densitycond} in the proof of Corollary \ref{Cor:densShifted}.
    
     To show \eqref{newcond} we just note that 
     the discriminant of $f(x)$ is equal to the discriminant of $f(x+a)$
     for any $a\in \ri$. Hence, the same proof as in Corollary \ref{cor:eisenstein} works to verify \eqref{newcond} for any $\alpha>0$.
    
    In order to verify \eqref{newcond2} and \eqref{newcond3}, we use Theorem \ref{shiftedchar} to obtain
    \begin{align*}
        \bigcap_{j=1}^n \widetilde{U}_{\mathfrak{p}_j} \cap O(H, \mathbb{E})^{d+1}
        &= \bigcap_{j=1}^n \widetilde{\Gamma}_{\mathfrak{p}_j}^d \cap O(H, \mathbb{E})^{d+1} \\
        &= \bigcap_{j=1}^n \bigsqcup_{b_j + \mathfrak{p}_j \in \ri / \mathfrak{p}_j} \sigma_{b_j}^{-1} (\Gamma_{\mathfrak{p}_j}^d) \cap O(H, \mathbb{E})^{d+1} \\
        &= \bigcup_{b_1+\mathfrak{p}_1\in \ri/\mathfrak{p}_1, \dots, b_n+\mathfrak{p}_n \in \ri/\mathfrak{p}_n} \bigcap_{j=1}^n \sigma_{b_j}^{-1}(\Gamma_{\mathfrak{p}_j}^d) \cap O(H, \mathbb{E})^{d+1}.
    \end{align*}
    By the Chinese Remainder Theorem there exists $b\in \ri$ such that $b+ \mathfrak{p}_j = b_j + \mathfrak{p}_j$ for all $j\in \{1, \dots, n\}$ and thus, by Proposition \ref{binp}, we have $\sigma_{b_j}^{-1}(\Gamma_{\mathfrak{p}_j}^d) = \sigma_{b}^{-1}( \Gamma_{\mathfrak{p}_j}^d)$. This implies
    \begin{align*}
        \bigcap_{j=1}^n \sigma_{b_j}^{-1} (\Gamma_{\mathfrak{p}_j}^d) = \bigcap_{j=1}^n \sigma_b^{-1} (\Gamma_{\mathfrak{p}_j}^d) = \sigma_b^{-1} \left( \bigcap_{j=1}^n \Gamma_{\mathfrak{p}_j}^d \right).
    \end{align*}
    Combining the previous identities yields
    \begin{align*}
        \bigcap_{j=1}^n \widetilde{U}_{\mathfrak{p}_j} \cap O(H, \mathbb{E})^{d+1} 
        = \bigcup_{b + \bigcap_{j=1}^n \mathfrak{p}_j \in \ri / \bigcap_{j=1}^n \mathfrak{p}_j} \sigma_b^{-1} \left( \bigcap_{j=1}^n \Gamma_{\mathfrak{p}_j}^{d+1} \right) \cap O(H, \mathbb{E})^{d+1}.
    \end{align*}
    Hence, we choose $\alpha= 1/(nk)$ and denote by $\varphi: \ri \rightarrow \mathbb{Z}^k$ again the isomorphism induced by $\mathbb{E}$. By the same argument as in the proof of Corollary \ref{cor:eisenstein} we get that for $\mathfrak{p}_1, \dots, \mathfrak{p}_n \preceq H^\alpha$ we have $\operatorname{diam}(\varphi(\bigcap_{j=1}^n \mathfrak{p}_j)) \leq H$ and thus by the computation \eqref{auxEisenstein} we obtain
    \begin{align*}
        \frac{\left\vert \bigcap_{j=1}^n \widetilde{U}_{\mathfrak{p}_j} \cap O(H, \mathbb{E})^{d+1} \right\vert}{(2H)^{k(d+1)}}
        &\leq \left(\prod_{j=1}^n p_j^{\operatorname{deg}(\mathfrak{p}_j)} \right)  \frac{\left\vert \bigcap_{j=1}^n U_{\mathfrak{p}_j} \cap O(H, \mathbb{E})^{d+1} \right\vert}{(2H)^{k(d+1)}} \\
        &\leq \prod_{j=1}^n p_j^{-(d-1)\operatorname{deg}(\mathfrak{p}_j)}
        \leq \prod_{j=1}^n \frac{1}{p_j^2},
    \end{align*}
    where we used $d\geq 3$ to get the last inequality. Therefore, also \eqref{newcond2} and \eqref{newcond3} are verified. Thus, all the claims follow from Theorem \ref{Thm:highermoments}.
\end{proof}

\begin{remark}
    Note that Corollary \ref{cor:shiftEisenstein} does not extend to $d=2$. For $d=2$ the conditions of Theorem \ref{Thm:highermoments} are satisfied for no positive integer. In fact, the system does not even satisfy the weaker conditions of Theorem \ref{PoonenStollNumbfields} as $\sum_{\eta \in N_K} s_\eta $ diverges, as we saw in the proof of Corollary \ref{Cor:densShifted}.
\end{remark}

\section*{Acknowledgments}
The work of Giacomo Micheli is partially supported by the National Science Foundation grant number 2127742. The work of Severin Schraven is supported by NSERC of Canada.
The work of Simran Tinani is supported by armasuisse Science and Technology. The work of Violetta Weger is supported by the Swiss National Science Foundation grant number 195290.

\bibliographystyle{plain}
\bibliography{biblio}

\end{document}